\documentclass[11pt,reqno]{amsart}

\usepackage{amsmath}
\usepackage{amssymb}
\usepackage{amscd}
\usepackage{latexsym}
\usepackage{graphicx}
\usepackage{color}

\newcommand{\ecal}{\mathcal{E}}
\newcommand{\fcal}{\mathcal{F}}
\newcommand{\pcal}{\mathcal{P}}
\newcommand{\scal}{\mathcal{S}}
\newcommand{\ucal}{\mathcal{U}}

\newcommand{\wh}{\widehat}
\newcommand{\ispa}[1]{\langle \,#1 \,\rangle } 
\newcommand{\ol}{\overline}

\newcommand{\mb}{\mathbb}
\newcommand{\ve}{\varepsilon}

\newcommand{\dsp}{\displaystyle}

\newtheorem{mthm}{{\sc Theorem}}
\newtheorem{thm}{{\sc Theorem}}[section]
\newtheorem{cor}[thm]{{\sc Corollary}}
\newtheorem{lem}[thm]{{\sc Lemma}}
\newtheorem{prop}[thm]{{\sc Proposition}}
\newtheorem{defin}[thm]{{\sc Definition}}
\theoremstyle{definition}

\begin{document}

\title[Grover and Fourier walks]
{Eigenvalues of quantum walks of\\ Grover and Fourier types}
\author{Takashi Komatsu}
\address{Department of Applied Mathematics, Faculty of Engineering, Yokohama National University, 
79-5 Tokiwadai, Hodogaya, Yokohama, 240-8501, Japan}
\email{komatsu-takashi-fn@ynu.ac.jp}
\author{Tatsuya Tate}
\address{Mathematical Institute, Graduate School of Sciences, Tohoku University, 
Aoba, Sendai 980-8578, Japan. }
\email{tate@m.tohoku.ac.jp}
\thanks{The second author is partially supported by JSPS Grant-in-Aid for Scientific Research (No. 25400068, No. 15H02055).}
\date{\today}

\renewcommand{\thefootnote}{\fnsymbol{footnote}}
\renewcommand{\theequation}{\thesection.\arabic{equation}}
\renewcommand{\labelenumi}{{\rm (\arabic{enumi})}}
\renewcommand{\labelenumii}{{\rm (\alph{enumii})}}
\numberwithin{equation}{section}

\begin{abstract}
A necessary and sufficient conditions for certain 
class of periodic unitary transition operators to have eigenvalues are given. 
Applying this, it is shown that Grover walks in any dimension has both of $\pm 1$ as eigenvalues 
and it has no other eigenvalues. It is also shown that the lazy Grover walks in 
any dimension has $1$ as an eigenvalue, and it has no other eigenvalues. 
As a result, a localization phenomenon occurs for these quantum walks. 
A general criterion for the existence of eigenvalues can be applied also to 
certain quantum walks of Fourier type. 
It is shown that the two-dimensional Fourier walk does not have 
eigenvalues and hence it is not localized at any point. 
Some other topics such as Grover walks on the triangular lattice, products and 
deformations of Grover walks are also discussed. 
\end{abstract}

\maketitle

\section{Introduction}\label{INTRO}

{\it Quantum walks}, originally introduced in \cite{ADZ}, 
have recently become popular among computer science, quantum physics, probability theory 
and discrete geometric analysis, and it is a generic word used to mean certain family of probability distributions 
on a parameter space of an orthonormal basis of a separable Hilbert space 
defined by powers of unitary operators called {\it unitary transition operators}. 
The readers may be referred to \cite{Ke}, \cite{Ko2} for historical backgrounds of quantum walks. 
One of peculiar properties of quantum walks is its {\it localization} property, which means that the probability 
that a particle is found at a point does not tend to zero as the time parameter goes to infinity. 
Even one-dimensional quantum walks (\cite{IKS}, \cite{IK}) can have the localization property. 
The localization of quantum walks is closely related to the existence of 
eigenvalues of unitary transition operators. Indeed, it is well-known that a quantum walk with an initial data 
has localized at a point with if the associated unitary transition operator has an eigenvalue. 
However, although there are many investigation on localization of quantum walks 
defined by specific unitary transition operators, 
it seems that few works are devoted to study it from a general viewpoint. 

The {\it periodic unitary transition operators} (PUTO for short) was introduced in \cite{Ta1} 
as unitary operators on the Hilbert space $\ell^{2}(\mb{Z}^{d},\mb{C}^{D})$, 
consisting of all $\mb{C}^{D}$-valued $\ell^{2}$-functions on $\mb{Z}^{d}$, 
commuting with the action of $\mb{Z}^{d}$ and having certain finite propagation property. 
Readers may be referred to \cite{Ta1} for a precise definition and properties of PUTO's. 
In the present paper, we consider a class of PUTO's defined as follows. 
Let $S \subset \mb{Z}^{d}$ be a finite set called a {\it set of steps} and let $\pcal=\{P_{\alpha}\}_{\alpha \in S}$ be 
a resolution of unity on $\mb{C}^{D}$, that is $\pcal=\{P_{\alpha}\}_{\alpha \in S}$ is a finite family of 
orthogonal projections on $\mb{C}^{D}$ satisfying 
\begin{equation}\label{resU1}
P_{\alpha}^{2}=P_{\alpha},\quad P_{\alpha}P_{\beta}=0\ \ (\alpha \neq \beta),\quad 
\sum_{\alpha \in S}P_{\alpha}=I, 
\end{equation}
where $I$ is the identity map on $\mb{C}^{D}$. 
Then the PUTO's we are going to consider are of the form 
\begin{equation}\label{PUTO1}
U(C)=U(S,\pcal,C)=\scal C,\quad 
\scal=\scal(S,\pcal)=\sum_{\alpha \in S} \tau^{\alpha} P_{\alpha},
\end{equation}
where, for $\alpha=(\alpha_{1},\ldots,\alpha_{d}) \in \mb{Z}^{d}$, 
the operator $\tau^{\alpha}$ is defined by 
\begin{equation}\label{ACT}
(\tau^{\alpha}f)(x)=f(x-\alpha) \qquad (f \in \ell^{2}(\mb{Z}^{d},\mb{C}^{D}),\ x \in \mb{Z}^{d}),  
\end{equation}
and $C$ is a $D \times D$ unitary matrix. 
Both of $\scal$ and $C$ are unitary operators on $\ell^{2}(\mb{Z}^{d},\mb{C}^{D})$, 
and they are called the {\it shift operator} and the {\it coin matrix}, respectively. 
(We note that $C$ is a $D \times D$ matrix. 
But it can be regarded as a bounded operator on $\ell^{2}(\mb{Z}^{d},\mb{C}^{D})$ in a natural way.) 
The {\it quantum walk} with the initial state $\phi$ and the time evolution $U(S,\pcal,C)$
is a family of probability distributions $\{p_{n}(\phi;x)\}_{x \in \mb{Z}^{d}}$ 
indexed by a discrete time parameter $n$ defined by
\[
p_{n}(\phi;x)=\|U(S,\pcal,C)^{n}(\delta_{0} \otimes \phi)(x)\|_{\mb{C}^{D}}^{2},  
\]
where, $\phi$ is a unit vector in $\mb{C}^{D}$, $n$ is a positive integer and, 
for $z \in \mb{Z}^{d}$, $\delta_{z} \otimes \phi \in \ell^{2}(\mb{Z}^{d},\mb{C}^{D})$ is defined by 
\[
(\delta_{z} \otimes \phi)(x)=
\begin{cases}
\phi & \mbox{($x=z$)}, \\
0 & \mbox{(otherwise)}. 
\end{cases}
\]
We note that, since the operator $U(S,\pcal,C)$ is commutative 
with the action of $\mb{Z}^{d}$ on $\ell^{2}(\mb{Z}^{d},\mb{C}^{D})$, 
it is enough to consider the origin as the initial position. 
As an example of operators of the form $\eqref{PUTO1}$, 
let us explain usual homogeneous quantum walks with constant coin matrices. 
In the following, let $\{\pmb{e}_{1},\ldots,\pmb{e}_{D}\}$ be the standard orthonormal basis on $\mb{C}^{D}$ 
and let $P_{j}$ be the orthogonal projection onto the $1$-dimensional subspace $\mb{C}\pmb{e}_{j}$ in $\mb{C}^{D}$. 

We set $D=2d$ and $S=S_{{\rm std}}=\{\pm u_{1},\ldots, \pm u_{d}\}$ with the standard orthonormal 
basis $\{u_{1},\ldots,u_{d}\}$ of $\mb{Z}^{d}$. 
We define $\pcal=\pcal_{{\rm std}}=\{P_{\alpha}\}_{\alpha \in S_{{\rm std}}}$ by 
\begin{equation}\label{std1}
P_{u_{j}}=P_{2j-1},\quad P_{-u_{j}}=P_{2j} \quad (j=1,\ldots,d). 
\end{equation}
Then the unitary operator $U(S_{{\rm std}}, \pcal_{{\rm std}}, C)$ has the form
\begin{equation}\label{UQW}
U(S_{{\rm std}}, \pcal_{{\rm std}}, C)=\sum_{j=1}^{d}\left(
\tau_{j}P_{2j-1}C+\tau_{j}^{-1}P_{2j}C
\right),
\end{equation}
where $\tau_{j}=\tau^{u_{j}}$. 
The operator $\eqref{UQW}$ is 
what we call the time evolution of {\it homogeneous quantum walks} with a coin matrix $C$. 

The class of PUTO's of the form $\eqref{PUTO1}$ also contains ``lazy quantum walks'', originally 
defined and investigated in \cite{IKS}, \cite{IK} as a 1-dimensional 3-state model. To explain it, 
we set $D=2d+1$ and we define the set $S_{{\rm lazy}}$ of steps by $S_{{\rm lazy}}=S_{{\rm std}} \cup \{0\}$ and define 
the resolution of unity $\pcal_{{\rm lazy}}=\{P_{\alpha}\}_{\alpha \in S_{{\rm lazy}}}$ by 
\begin{equation}\label{lazy1}
P_{u_{j}}=P_{j},\quad P_{0}=P_{d+1},\quad P_{-u_{j}}=P_{d+1+j} \quad (j=1,\ldots,d). 
\end{equation}
The corresponding PUTO $U(S_{{\rm lazy}},\pcal_{{\rm lazy}},C)$ has the form
\begin{equation}\label{lazyD}
U(S_{{\rm lazy}},\pcal_{{\rm lazy}},C)=P_{0}C+
\sum_{j=1}^{d}\left(
\tau_{j}P_{j}C+\tau_{j}^{-1}P_{d+1+j}C
\right). 
\end{equation}

There are two kind of {\it Grover walks} (\cite{W}, \cite{MBSS}), which named after Grover's investigation on 
quantum search algorithm \cite{G}, one of which is defined by a shift operator called flip-flop shift, 
and another is defined as $U(S_{{\rm std}}, \pcal_{{\rm std}}, G_{2d})$ and $U(S_{{\rm lazy}}, \pcal_{{\rm lazy}}, G_{2d+1})$ 
where, for general positive integer $D$, a $D \times D$ coin matrix $G_{D}$ called {\it Grover matrix} 
is defined by 
\begin{equation}\label{GroC}
G_{D}=\frac{2}{D}J-I,
\end{equation}
with the $D \times D$ matrix $J$ all of whose components are $1$. 

The time evolution of the {\it Fourier walk} (\cite{MBSS}) is defined, for the case of $D=2d=4$, 
as $U(S_{{\rm std}}, \pcal_{{\rm std}}, F_{4})$ where, for positive integer $D \geq 2$, 
the coin matrix $F_{D}$ is the $D \times D$ unitary matrix, defining 
the discrete Fourier transform, given by 
\begin{equation}\label{DF1}
F_{D}=
\frac{1}{\sqrt{D}}
\begin{pmatrix}
1 & 1 & 1 & \cdots & 1 \\
1 & q & q^{2} & \cdots & q^{D-1} \\
\vdots & \vdots & \vdots & \ddots & \vdots \\
1 & q^{D-2} & q^{2(D-2)} & \cdots & q^{(D-1)(D-2)} \\
1 & q^{D-1} & q^{2(D-1)} & \cdots & q^{(D-1)(D-1)}
\end{pmatrix}, 
\end{equation}
where $q=e^{2\pi \sqrt{-1}/D}$. 
The coin matrix $F_{D}$ satisfies $F_{D}^{4}=I$. 
Taking these operators into account, it would be natural to give the following definition. 
\begin{defin}\label{GF1}
Let $C$ be a $D \times D$ unitary matrix which is not a scalar multiple of the identity matrix. 
The quantum walks defined by a PUTO, $U(S,\pcal,C)$, given in $\eqref{PUTO1}$, or $U(S,\pcal,C)$ itself, 
is said to be of {\it Grover type} if the coin matrix $C$ satisfies $C^{2}=I$. 
A PUTO, $U(S,\pcal,C)$, of Grover type is said to be of {\it reflection type} if 
the multiplicity of eigenvalue $1$ of the coin matrix $C$ is $1$. 
A PUTO, $U(S,\pcal,C)$, is said to be of {\it Fourier type} if $C$ satisfies $C^{4}=I$. 
\end{defin}
We remark that a coin matrix of reflection type is negative of a reflection in a usual sense. 
In the following sections, we give some of criterions for $U(S,\pcal,C)$ to have eigenvalues and  
hence for corresponding quantum walks with an initial state to be localized at a point. 
The following theorems will be proved by applying theorems which will be obtained in the subsequent sections. 
\begin{mthm}\label{Main3}
The Fourier walk $U(S_{{\rm std}}, \pcal_{{\rm std}}, F_{4})$ on $\ell^{2}(\mb{Z}^{2},\mb{C}^{4})$ has no eigenvalues. 
Hence the corresponding quantum walk $\{p_{n}(\phi;x)\}_{x \in \mb{Z}^{2}}$ satisfies 
\[
\lim_{n \to \infty}p_{n}(\phi;x)=0
\] 
for any initial state $\phi \in \mb{C}^{4}$ and any point $x \in \mb{Z}^{2}$. 
\end{mthm}

\begin{mthm}\label{Main1}
Both of $\pm 1$ are eigenvalues of 
the Grover walks $U(S_{{\rm std}}, \pcal_{{\rm std}}, G_{2d})$ on $\ell^{2}(\mb{Z}^{d},\mb{C}^{2d})$ 
with $d \geq 2$ and there are no other eigenvalues. 
Let $\Pi_{\pm}$ denote the orthogonal projections 
onto the eigenspaces of eigenvalue $\pm 1$, respectively. Then we have 
\begin{equation}\label{LTA1}
\lim_{N \to \infty}\frac{1}{N}\sum_{n=1}^{N} p_{n}(\phi;x)=\|\Pi_{+}(\delta_{0} \otimes \phi)(x)\|_{\mb{C}^{2d}}^{2}+
\|\Pi_{-}(\delta_{0} \otimes \phi)(x)\|_{\mb{C}^{2d}}^{2} 
\end{equation}
for any $\phi \in \mb{C}^{2d}$ and $x \in \mb{Z}^{d}$. 
Therefore, there exists a vector $\phi_{o} \in \mb{C}^{2d}$ such that 
the corresponding quantum walk $\{p_{n}(\phi_{o};x)\}_{x \in \mb{Z}^{d}}$ is localized at a point $x=x_{o} \in \mb{Z}^{d}$. Namely 
we have 
\begin{equation}\label{loc2}
\limsup_{n \to \infty} p_{n}(\phi_{o};x_{o})>0
\end{equation}
for a vector $\phi_{o} \in \mb{C}^{2d}$ and a point $x_{o} \in \mb{Z}^{d}$. 
\end{mthm}

\begin{mthm}\label{Main2}
The lazy Grover walks $U(S_{{\rm lazy}}, \pcal_{{\rm lazy}}, G_{2d+1})$ on $\ell^{2}(\mb{Z}^{d},\mb{C}^{2d+1})$ 
with $d \geq 1$ has just one eigenvalue $1$, and there are no other eigenvalues. 
Let $\Pi_{+}$ denote the orthogonal projection onto the eigenspace of eigenvalue $1$. 
Then we have 
\[
\lim_{n \to \infty}p_{n}(\phi;x)=\|\Pi_{+}(\delta_{0} \otimes \phi)(x)\|_{\mb{C}^{2d+1}}^{2} 
\]
for any $\phi \in \mb{C}^{2d+1}$ and any $x \in \mb{Z}^{d}$. 
\end{mthm}

In Theorem $\ref{Main2}$, the projection $\Pi_{+}$ onto the eigenspace 
of $U(S_{{\rm lazy}}, \pcal_{{\rm lazy}}, G_{2d+1})$ with eigenvalue $+1$ 
is, in principle, computable. Indeed we have 
\begin{equation}\label{EGF1}
\Pi_{+}(\delta_{0} \otimes \phi)(x)=
\int_{T^{d}} z^{-x} \ispa{\phi,w_{o}(z)}_{\mb{C}^{2d+1}}w_{o}(z)\,d\nu(z) \quad (x \in \mb{Z}^{d}), 
\end{equation}
where $\nu$ is the normalized Lebesgue measure on the $d$-dimensional torus $T^{d}$ 
and the function $w_{o} \in L^{\infty}(T^{d},\mb{C}^{2d+1})$ is given by 
\begin{equation}\label{EGF2}
w_{o}(z) = \frac{1}{D(z)} \left(
\frac{1}{2}\pmb{e}_{d+1} +\sum_{j=1}^{d}\frac{1}{1+z_{j}}
\left(
z_{j}\pmb{e}_{j}+\pmb{e}_{d+1+j}
\right)
\right), \quad 
D(z) = \left(
\frac{1}{4} +2\sum_{j=1}^{d}\frac{1}{|1+z_{j}|^{2}}
\right)^{1/2}.
\end{equation}
The formulas $\eqref{EGF1}$, $\eqref{EGF2}$ are direct consequences of Theorem $\ref{EGS1}$ in Section $\ref{GTYPE}$.
Theorems $\ref{Main3}$, $\ref{Main1}$, $\ref{Main2}$ might be proved in an straightforward way. 
For example, it might be possible 
to prove Theorem $\ref{Main2}$ by computing the characteristic polynomials. 
However, our approach is to prepare a necessary and sufficient conditions for the class of PUTO 
given in $\eqref{PUTO1}$ (Propositions $\ref{gen1}$), of Fourier type (Theorem $\ref{DFT}$) 
or of Grover type (Theorem $\ref{Gro1}$) to have eigenvalues and apply it to obtain these theorems. 
An advantage of this rather indirect method is that properties making operators to have eigenvalues 
become very clear. In fact Theorems $\ref{GroT}$, $\ref{EVN1}$, $\ref{EVNN}$ give 
very useful sufficient condition for the existence and non-existence of eigenvalues. 

It seems that quantum walks of Fourier type 
has a structure richer than quantum walks of Grover type. 
Indeed, it would be rather hard to apply our Theorem $\ref{DFT}$ for 
Fourier walks in higher dimension. 
In contrast, some aspects about localization for quantum walks of Grover types 
in high dimension have been clarified to some extent. 
However, dynamical aspects are not investigated in the present paper. 
It would be interesting and important to consider asymptotic behavior of 
quantum walks of reflection type and their eigenfunctions. 

In Section $\ref{GEN}$, we collect some general facts from \cite{Ta1} and 
give a general necessary and sufficient condition for $U(S,\pcal,C)$ to have an eigenvalue 
under the assumption that the origin $0 \in \mb{Z}^{d}$ is not contained in the set $S$ of steps. 
This is utilized to give a necessary and sufficient condition for $U(S,\pcal,C)$ of 
Fourier type in Section $\ref{FTYPE}$. 
In particular, Theorem $\ref{Main3}$ on Fourier walk $U(S_{{\rm std}}, \pcal_{{\rm std}}, F_{4})$ 
will be proved in Section $\ref{FW2D}$. 
A necessary and sufficient condition for $U(S,\pcal,C)$ of Grover type without assuming $0 \not\in S$ 
will be given in Section $\ref{GTYPE}$. We also give a sufficient condition on PUTO's defined 
by a product of certain two PUTO's of the form $\eqref{PUTO1}$ in Section $\ref{PRO}$. 
In Section $\ref{EXA}$, we give proofs of Theorems $\ref{Main1}$, $\ref{Main2}$, an example 
about a Grover walk on the triangular lattice, and discuss about a one-parameter deformation of Grover walks.

\section{Quantum walks of Fourier type}\label{NFT}

In this section, properties of the PUTO $U(C)=U(S,\pcal,C)$ given in $\eqref{PUTO1}$, in particular 
a necessary and sufficient conditions for $U(C)$ to have eigenvalues, are prepared, 
and then it will be applied to the quantum walks of Fourier type. 

\subsection{Localization with general coin matrices}\label{GEN}

We denote $\fcal:\ell^{2}(\mb{Z}^{d},\mb{C}^{D}) \to L^{2}(T^{d},\mb{C}^{D})$ 
the Fourier transform given by 
\[
(\fcal f)(z)=\sum_{x \in \mb{Z}^{d}} f(x)z^{x} \quad (f \in \ell^{2}(\mb{Z}^{d},\mb{C}^{D}),\, z \in T^{d}), 
\]
where, for $z=(z_{1},\ldots,z_{d}) \in T^{d}$ and $x=(x_{1},\ldots,x_{d}) \in \mb{Z}^{d}$, 
we set $z^{x}=z_{1}^{x_{1}} \cdots z_{d}^{x_{d}}$. 
The Fourier transform $\fcal$ is a unitary operator whose inverse $\fcal^{-1}=\fcal^{*}$ is given by 
\[
(\fcal^{-1}g)(x)=\int_{T^{d}} z^{-x} g(z)\,d\nu(z), 
\]
where $d\nu$ is the normalized Lebesgue measure on $T^{d}$. 
For the operator $U(S,\pcal,C)$ with the coin matrix $C$, 
the unitary operator 
\begin{equation}\label{FOP1}
\ucal(S,\pcal,C)=\fcal U(S,\pcal,C) \fcal^{-1}
\end{equation}
on $L^{2}(T^{d},\mb{C}^{D})$ has the form 
\begin{equation}\label{FOP2}
\ucal(S,\pcal,C)g (z)=\wh{C}(z)g(z) \quad (g \in L^{2}(T^{d},\mb{C}^{D}),\ z \in T^{d}), 
\end{equation}
where, for $z \in T^{d}$, $\wh{C}(z)$ is a $D \times D$ unitary matrix defined by 
\begin{equation}\label{MV1}
\wh{C}(z)=V(z)C,\quad V(z)=\sum_{\alpha \in S}z^{\alpha}P_{\alpha}. 
\end{equation}
The following is well-known (see \cite{Ta1} for instance). 
\begin{lem}\label{twist1}
The unitary operator $U(S,\pcal,C)$ has an eigenvalue $\omega$ if and only if 
$\wh{C}(z)$ has an eigenvalue $\omega$ for any $z \in T^{d}$. 
\end{lem}
The next lemma is also well-known. But, its proof is given here for convenience and completeness. 
\begin{lem}\label{PL1}
Suppose that $U(C)=U(S,\pcal,C)$ has an eigenvalue. 
Then there exists a non-zero vector $\phi \in \mb{C}^{D}$ such that 
the corresponding quantum walk $\{p_{n}(\phi;x)\}_{x \in \mb{Z}^{d}}$ with initial state $\phi$ 
is localized at a point $x=x_{o} \in \mb{Z}^{d}$. 
\end{lem}
\begin{proof}
Suppose that $U(C)$ has an eigenvalue. Let $\Pi_{\lambda}$ be the orthogonal projection 
onto the eigenspace corresponding to the eigenvalue $\lambda$. 
We note that the functions of the form $\delta_{y} \otimes \phi$ with $y \in \mb{Z}^{d}$ 
and $\phi \in \mb{C}^{D}$ spanns $\ell^{2}(\mb{Z}^{d},\mb{C}^{D})$, and hence 
there exists a vector $\phi_{o} \in \mb{C}^{D}$, points $y,z \in \mb{Z}^{d}$ and an eigenvalue $\lambda$ such 
that $\Pi_{\lambda}(\delta_{y} \otimes \phi_{o})(z) \neq 0$. Since $U(C)$ is commutative with the action $\eqref{ACT}$ 
of $\mb{Z}^{d}$ on $\ell^{2}(\mb{Z}^{d},\mb{C}^{D})$, $\Pi_{\lambda}$ also commutative with this action of $\mb{Z}^{d}$. 
Thus, we have 
\[
0 \neq \Pi_{\lambda}(\delta_{y} \otimes \phi_{o})(z)=
\Pi_{\lambda}(\delta_{0} \otimes \phi_{o})(x_{o})\ \ \ (x_{o}=z-y). 
\]
Then by the Wiener formula, we see 
\begin{equation}\label{wie}
\lim_{N \to \infty}\frac{1}{N}\sum_{n=1}^{N} p_{n}(\phi;x_{o})=
\sum_{\omega\,:\,\mbox{{\scriptsize eigenvalue of $U(C)$}}} \|(\Pi_{\omega}(\delta_{0} \otimes \phi))(x_{o})\|_{\mb{C}^{D}}^{2}>0. 
\end{equation}
Therefore, $\eqref{loc2}$ holds for $\phi_{o}$ at $x=x_{o}$. 
\end{proof}

We denote by $\sigma(A)$ the set of eigenvalues of a matrix $A$. 
For an eigenvalue $\lambda$ of the coin matrix $C$, we denote $\pi_{\lambda}$ 
the projection onto the eigenspace $\ecal(\lambda)$ of $C$ corresponding to $\lambda$, 
and let $\pi_{\lambda}^{\perp}$ denote the projection onto the orthogonal complement $\ecal(\lambda)^{\perp}$ of 
$\ecal(\lambda)$. When $\lambda \not\in \sigma(C)$, we set $\ecal(\lambda)=\{0\}$ and $\pi_{\lambda}=0$. 

{\it In the rest of this section, we assume that the set $S$ of steps does not contain the origin}. 
In this case, the operator $U(C)=U(S,\pcal,C)$ do not have a ``lazy'' part $P_{0}$. 
We define a subset $L$ in $T^{d}$ by 
\[
L=\bigcup_{\alpha \in S} \{z \in T^{d} \mid z^{\alpha}=1\}. 
\]
Since, for $\alpha \in S$, the function $T^{d} \ni z \mapsto z^{\alpha} \in \mb{C}$ is a non-trivial 
homomorphism on the abelian Lie group $T^{d}$, $L$ is a finite union of closed submanifolds in $T^{d}$ (see \cite{BD} for instance). 
For each $z \in T^{d} \setminus L$ and $\omega \in \sigma(C)$, we define a linear map $\eta(\omega;z)$ on $\mb{C}^{D}$ by 
\begin{equation}\label{FUTO1}
\eta(\omega;z)=\sum_{\alpha \in S}\sum_{\lambda \in \sigma(C) \setminus \{\omega\}} 
\frac{1-\omega^{-1}\lambda z^{\alpha}}{1-z^{\alpha}} P_{\alpha}\pi_{\lambda}. 
\end{equation}
It is clear that $\eta(\omega;z)\ecal(\omega)=\{0\}$. Furthermore, we see that 
\begin{equation}\label{FUTO2}
\eta(\omega;z)=\pi_{\omega}^{\perp}
+\sum_{\alpha \in S}
\frac{z^{\alpha}}{1-z^{\alpha}} P_{\alpha}(I -\omega^{-1}C)\pi_{\omega}^{\perp}
=\sum_{\alpha \in S}\frac{1}{1-z^{\alpha}} P_{\alpha} (I-z^{\alpha}\omega^{-1}C) \pi_{\omega}^{\perp}.
\end{equation}
We remark that if $U(C)$ has an eigenvalue $\omega$, 
then $\omega$ is also an eigenvalue of the unitary matrix $C$. 
The following proposition will not be used until Section $\ref{FTYPE}$. 
\begin{prop}\label{gen1}
Suppose that $S$ does not contain the origin. 
Suppose also that $C$ is not a scalar multiple of the identity matrix. Let $\omega \in \sigma(C)$. 
Then, the following statements are equivalent. 
\begin{enumerate}
\item $\omega$ is an eigenvalue of the unitary operator $U(C)$. 
\item For each $z \in T^{d} \setminus L$, there exists a non-zero vector 
$\psi \in \ecal(\omega)^{\perp}$ such that $\eta(\omega;z)\psi \in \ecal(\omega)$. 
\end{enumerate}
\end{prop}
\begin{proof}
First we assume that the statement $(1)$ holds. Then, for any $z \in T^{d}$, 
$\omega$ is an eigenvalue of the unitary matrix $\wh{C}(z)$ given in $\eqref{MV1}$. 
We fix a point $z \in T^{d} \setminus L$. Let us take an eigenvector $0 \neq \phi \in \mb{C}^{D}$ of $\wh{C}(z)$ 
with the eigenvalue $\omega$. 
Then, writing $\phi=\pi_{\omega}\phi+\pi_{\omega}^{\perp}\phi$ and using $\eqref{resU1}$, we have 
\[
\begin{split}
0 & = \wh{C}(z)\phi -\omega \phi=\sum_{\alpha \in S} z^{\alpha}P_{\alpha}C(\pi_{\omega} \phi +\pi_{\omega}^{\perp}\phi)
-\omega \pi_{\omega}\phi -\omega \pi_{\omega}^{\perp}\phi \\
& = -\sum_{\alpha \in S} P_{\alpha}
\left[
\omega(1-z^{\alpha}) \pi_{\omega} 
+(\omega -z^{\alpha}C)\pi_{\omega}^{\perp}
\right]\phi. 
\end{split}
\]
Applying $P_{\alpha}$ to the above equation, we have 
\[
P_{\alpha}\pi_{\omega}\phi =-\frac{1}{1-z^{\alpha}} P_{\alpha}(1-z^{\alpha}\omega^{-1}C) \pi_{\omega}^{\perp}\phi.
\]
Summing this over all $\alpha \in S$ and using $\eqref{FUTO2}$, we obtain
\begin{equation}\label{eigenE1}
\pi_{\omega}\phi=-\eta(\omega;z)\phi=-\eta(\omega;z)\pi_{\omega}^{\perp}\phi.
\end{equation}
We see, by $\eqref{eigenE1}$, that $\pi_{\omega}^{\perp}\phi \neq 0$ and 
$\eta(\omega;z)\pi_{\omega}^{\perp}\phi \in \ecal(\omega)$. 
This implies (2). 

Conversely, we suppose that the statement $(2)$ holds. We fix $z \in T^{d} \setminus L$. 
We take a non-zero $\psi \in \ecal(\omega)^{\perp}$ such that $\eta(\omega;z)\psi \in \ecal(\omega)$. 
Here we remark that $\eta(\omega;z)\psi$ could be zero. 
We set $\phi=\psi-\eta(\omega;z)\psi$. Then $\phi \neq 0$ and it follows from $\eqref{FUTO2}$ that 
\[
\phi=(I-\eta(\omega;z))\psi=-\sum_{\alpha \in S} \frac{z^{\alpha}}{1-z^{\alpha}} P_{\alpha}(I-\omega^{-1}C)\psi.
\]
A direct computation 
using the property that $\eta(\omega;z)\psi \in \ecal(\omega)$ and the above equation shows that 
$\wh{C}(z)\phi=\omega \phi$, which shows that $\omega$ is an eigenvalue of $\wh{C}(z)$ for any $z \in T^{d} \setminus L$. 
Since the characteristic polynomial of $\wh{C}(z)$ is continuous in $z \in T^{d}$ and $T^{d} \setminus L$ is an open dense 
set in $T^{d}$, $\omega$ is an eigenvalue of $\wh{C}(z)$ for any $z \in T^{d}$, and hence we have (1). 
\end{proof}

\subsection{Quantum walks of Fourier type}\label{FTYPE}

In this subsection, we utilize Proposition $\ref{gen1}$ to analyze quantum walks of Fourier type.  
Let $U(S,\pcal,C)$ be a PUTO of Fourier type. Thus the coin matrix $C$ 
is not a scalar multiple of the identity matrix and satisfies $C^{4}=I$. 
In this case $\sigma(C) \subset \{\pm 1,\pm i\}$. 
As a convention, we set $\pi_{\lambda}=0$ when $\lambda \not\in \sigma(C)$. 
Thus there are possibly four spectral projections $\pi_{1}$, $\pi_{i}$, $\pi_{-1}$, $\pi_{-i}$ of $C$. 
The linear map $\eta(\omega;z)$ defined by $\eqref{FUTO1}$ for $z \in T^{d} \setminus E$ is written as 
\begin{equation}\label{DF2}
\begin{split}
\eta(1;z) & = M(z)\pi_{i} +K(z)\pi_{-1} +L(z)\pi_{-i},\quad 
\eta(i;z)  = L(z)\pi_{1} + M(z) \pi_{-1} + K(z) \pi_{-i}, \\
\eta(-1;z) & = K(z) \pi_{1} +L(z) \pi_{i} +M(z) \pi_{-i},\quad 
\eta(-i;z)  = M(z) \pi_{1} + K(z) \pi_{i} +L(z) \pi_{-1}, 
\end{split}
\end{equation}
where the matrices $K(z)$, $L(z)$, $M(z)$ are given by 
\begin{equation}\label{DF3}
K(z)=\sum_{\alpha \in S} \frac{1+z^{\alpha}}{1-z^{\alpha}} P_{\alpha},\quad 
L(z)=\sum_{\alpha \in S} \frac{1+iz^{\alpha}}{1-z^{\alpha}} P_{\alpha},\quad 
M(z)=\sum_{\alpha \in S} \frac{1-iz^{\alpha}}{1-z^{\alpha}} P_{\alpha}.
\end{equation}
Thus, in this case, Proposition $\ref{gen1}$ can be rewritten in the following form. 
\begin{thm}\label{DFT}
Suppose that the PUTO $U(C)=U(S,\pcal,C)$ is of Fourier type. 
Then, the following statements {\rm (a-$k$)} and {\rm (b-$k$)} are equivalent to each other for each $k=0,1,2,3$. 
\begin{itemize}
\item[{\rm (a-$k$)}] $U(C)$ has an eigenvalue $i^{k}$. 
\item[{\rm (b-$k$)}] For any $z \in T^{d} \setminus L$, there exists a vector $\psi \in \ecal(i^{k})^{\perp}$ such that 
$\psi \neq 0$ and $\eta(i^{k};z) \psi \in \ecal(i^{k})$. 
\end{itemize}
Here $\eta(i^{k};z)$ are given by $\eqref{DF2}$, $\eqref{DF3}$.
\end{thm}
For example, when $-i$ is not the eigenvalues of $C$, 
the formula $\eqref{DF2}$ is simplified since $\pi_{-i}=0$. 
This is the case where $C=F_{4}$ defined in $\eqref{DF1}$ for $D=4$, 
although in higher dimension the coin matrix $F_{D}$ has definitely four eigenvalues as computed in \cite{Ma}. 

\subsection{Two dimensional Fourier walk}\label{FW2D}
Let us use Theorem $\ref{DFT}$ to prove Theorem $\ref{Main3}$. 
The two dimensional Fourier walk is $U(F_{4})=U(S_{{\rm std}}, \pcal_{{\rm std}}, F_{4})$ 
acting on $\ell^{2}(\mb{Z}^{2},\mb{C}^{4})$. 
The eigenvalues of $F_{4}$ are $1$, $-1$, $i$ and their multiplicities are $2$, $1$, $1$, respectively. 
An orthonormal basis of eigenvectors are given by 
\[
\frac{1}{\sqrt{2}}
\begin{bmatrix}
1 \\
0 \\
1 \\
0
\end{bmatrix},\quad 
\frac{1}{\sqrt{4}}
\begin{bmatrix}
1 \\
1 \\
-1 \\
1
\end{bmatrix}\ \ (\mbox{eigenvalue $1$}),\quad 
\frac{1}{\sqrt{2}}
\begin{bmatrix}
0 \\
1 \\
0 \\
-1
\end{bmatrix}\ \ (\mbox{eigenvalue $i$}), \quad
\frac{1}{\sqrt{4}}
\begin{bmatrix}
-1 \\
1 \\
1 \\
1
\end{bmatrix}\ \ (\mbox{eigenvalue $-1$}),
\]
and the eigenspaces are written as 
\[
\begin{split}
\ecal(1) & = \{\,\,\!\!^{t}[s+t,t,s-t,t] \in \mb{C}^{4} \mid s,t \in \mb{C}\},\\
\ecal(i) & = \{\,\,\!\!^{t}[0,s,0,-s] \in \mb{C}^{4} \mid s \in \mb{C}\}, \\ 
\ecal(-1)& = \{\,\,\!\!^{t}[-s,s,s,s] \in \mb{C}^{4} \mid s \in \mb{C}\}.
\end{split}
\]
Their orthogonal complements are
\[
\begin{split}
\ecal(1)^{\perp} & = \{\,\,\!\!^{t}[a+b,-2a,-(a+b),-2b] \in \mb{C}^{4} \mid a,b \in \mb{C}\},\\
\ecal(i)^{\perp} & = \{\,\,\!\!^{t}[a,b,c,b] \in \mb{C}^{4} \mid a,b,c \in \mb{C}\}, \\ 
\ecal(-1)^{\perp}& = \{\,\,\!\!^{t}[a+b+c,a,b,c] \in \mb{C}^{4} \mid a,b,c \in \mb{C}\}.
\end{split}
\]
The projections $\pi_{\pm 1}$, $\pi_{i}$ are given by 
\[
\begin{gathered}
\pi_{1}=\frac{1}{4}
\begin{bmatrix}
3 & 1 & 1 & 1 \\
1 & 1 & -1 & 1 \\
1 & -1 & 3 & -1 \\
1 & 1 & -1 & 1 
\end{bmatrix},\quad
\pi_{i}=\frac{1}{2}
\begin{bmatrix}
0 & 0 & 0 & 0 \\
0 & 1 & 0 & -1 \\
0 & 0 & 0 & 0 \\
0 & -1 & 0 & 1
\end{bmatrix},\quad
\pi_{-1}=\frac{1}{4}
\begin{bmatrix}
1 & -1 & -1 & -1 \\
-1 & 1 & 1 & 1 \\
-1 & 1 & 1 & 1 \\
-1 & 1 & 1 & 1
\end{bmatrix}.
\end{gathered}
\]
The matrices $K(z)$, $L(z)$, $M(z)$ are given by 
\[
\begin{gathered}
K(z)=
\begin{bmatrix}
\frac{1+z_{1}}{1-z_{1}} & 0 & 0 & 0 \\
0 & -\frac{1+z_{1}}{1-z_{1}} & 0 & 0 \\
0 & 0 & \frac{1+z_{2}}{1-z_{2}} & 0 \\
0 & 0 & 0 & -\frac{1+z_{2}}{1-z_{2}}
\end{bmatrix},\\
L(z)=
\begin{bmatrix}
\frac{1+iz_{1}}{1-z_{1}} & 0 & 0 & 0 \\
0 & -\frac{i+z_{1}}{1-z_{1}} & 0 & 0 \\
0 & 0 & \frac{1+iz_{2}}{1-z_{2}} & 0 \\
0 & 0 & 0 & -\frac{i+z_{2}}{1-z_{2}}
\end{bmatrix},\quad 
M(z)=
\begin{bmatrix}
\frac{1-iz_{1}}{1-z_{1}} & 0 & 0 & 0 \\
0 & \frac{i-z_{1}}{1-z_{1}} & 0 & 0 \\
0 & 0 & \frac{1-iz_{2}}{1-z_{2}} & 0 \\
0 & 0 & 0 & \frac{i-z_{2}}{1-z_{2}}
\end{bmatrix}.
\end{gathered}
\]
Thus, the matrix $\eta(i^{k};z)$ (in this case, we only need to consider the case $k=0,1,2$) is given by 
\[
\eta(1;z) =\frac{1}{4}
\begin{bmatrix}
\frac{1+z_{1}}{1-z_{1}} &-\frac{1+z_{1}}{1-z_{1}} &-\frac{1+z_{1}}{1-z_{1}} & -\frac{1+z_{1}}{1-z_{1}} \\[5pt]
\frac{1+z_{1}}{1-z_{1}} &\frac{2i-1-3z_{1}}{1-z_{1}} & -\frac{1+z_{1}}{1-z_{1}} & \frac{-2i-1+z_{1}}{1-z_{1}} \\[5pt]
-\frac{1+z_{2}}{1-z_{2}} &\frac{1+z_{2}}{1-z_{2}} &\frac{1+z_{2}}{1-z_{2}} & \frac{1+z_{2}}{1-z_{2}} \\[5pt]
\frac{1+z_{2}}{1-z_{2}} &\frac{-2i -1 +z_{2}}{1-z_{2}} & -\frac{1+z_{2}}{1-z_{2}} & \frac{2i-1-3z_{2}}{1-z_{2}}
\end{bmatrix},
\]
\[
\eta(i;z) =\frac{1}{4}
\begin{bmatrix}
\frac{4+2iz_{1}}{1-z_{1}} & \frac{2iz_{1}}{1-z_{1}} & \frac{2iz_{1}}{1-z_{1}} & \frac{2iz_{1}}{1-z_{1}} \\[5pt]
\frac{-2i}{1-z_{1}} & \frac{-2z_{1}}{1-z_{1}} & \frac{2i}{1-z_{1}} & \frac{-2z_{1}}{1-z_{1}} \\[5pt]
\frac{2iz_{2}}{1-z_{2}} & \frac{-2iz_{2}}{1-z_{2}} & \frac{4+2iz_{2}}{1-z_{2}} & \frac{-2iz_{2}}{1-z_{2}} \\[5pt]
\frac{-2i}{1-z_{2}} & \frac{-2z_{2}}{1-z_{2}} & \frac{2i}{1-z_{2}} & \frac{-2z_{2}}{1-z_{2}}
\end{bmatrix}, 
\]
\[
\eta(-1;z) =\frac{1}{4}
\begin{bmatrix}
\frac{3+3z_{1}}{1-z_{1}} & \frac{1+z_{1}}{1-z_{1}} & \frac{1+z_{1}}{1-z_{1}} & \frac{1+z_{1}}{1-z_{1}} \\[5pt]
-\frac{1+z_{1}}{1-z_{1}} & \frac{-2i-1-3z_{1}}{1-z_{1}} & \frac{1+z_{1}}{1-z_{1}} & \frac{2i-1+z_{1}}{1-z_{1}} \\[5pt]
\frac{1+z_{2}}{1-z_{2}} & -\frac{1+z_{2}}{1-z_{2}} & \frac{3+3z_{2}}{1-z_{2}} & -\frac{1+z_{2}}{1-z_{2}} \\[5pt]
-\frac{1+z_{2}}{1-z_{2}} & \frac{2i-1+z_{2}}{1-z_{2}} & \frac{1+z_{2}}{1-z_{2}} & \frac{-2i-1-3z_{2}}{1-z_{2}}
\end{bmatrix}. 
\]
We see 
\[
\begin{split}
\eta(1;z)
\begin{bmatrix}
a+b \\
-2a \\
-(a+b) \\
-2b
\end{bmatrix}
& =
\begin{bmatrix}
(1+z_{1})(a+b)/(1-z_{1}) \\
[(1-i+2z_{1})a+(1+i)b]/(1-z_{1}) \\
-(1+z_{2})(a+b)/(1-z_{2}) \\
[(1+i)a+(1-i+2z_{2})b]/(1-z_{2})
\end{bmatrix}, \\
\eta(i;z)
\begin{bmatrix}
a \\
b \\
c \\
b
\end{bmatrix}
& = 
\begin{bmatrix}
[iz_{1}(a+2b+c)+2a]/2(1-z_{1}) \\
[-i(a-c)-2z_{1}b]/2(1-z_{1}) \\
[iz_{2}(a-2b+c)+2c]/2(1-z_{2}) \\
[-i(a-c)-2z_{2}b]/2(1-z_{2})
\end{bmatrix},\\
\eta(-1;z)
\begin{bmatrix}
a+b+c\\
a\\
b\\
c
\end{bmatrix} 
& = 
\begin{bmatrix}
(1+z_{1})(a+b+c)/(1-z_{1}) \\
[-(1+i+2z_{1})a+(i-1)c]/2(1-z_{1}) \\
(1+z_{2})b/(1-z_{2}) \\
[(i-1)a-(i+1+2z_{2})c]/2(1-z_{2})
\end{bmatrix}.
\end{split}
\]
For $\psi=\,\!^{t}(a+b,-2a,-(a+b),-2b) \in \ecal(1)^{\perp}$, $\eta(1;z)\psi \in \ecal(1)$ if and only if 
\[
\begin{bmatrix}
-2i+(3+i)z_{1}-(1-i)z_{2}-2z_{1}z_{2} & 2i+(1-i)z_{1}-(3+i)z_{2}+2z_{1}z_{2} \\
2[-i+2z_{1}-(1-i)z_{2}-z_{1}z_{2}] & 2[i-(1+i)z_{2}+z_{1}z_{2}]
\end{bmatrix}
\begin{bmatrix}
a \\
b
\end{bmatrix}
=
\begin{bmatrix}
0 \\
0
\end{bmatrix}. 
\]
The determinant of the $2 \times 2$ matrix appeared in the above equals 
\[
-4(1-i)(z_{1}-z_{2})^{2}
\]
which does not vanish if $z_{1} \neq z_{2}$. 
Thus, when $z_{1} \neq z_{2}$, only $\psi=0$ satisfies $\eta(1;z)\psi \in \ecal(1)$, and hence 
$U(F_{4})$ does not have eigenvalue $1$. 
For $\psi=\,\!^{t}(a,b,c,b) \in \ecal(i)^{\perp}$, $\eta(i;z)\psi \in \ecal(i)$ if and only if 
\[
\begin{bmatrix}
i(2-z_{1}-z_{2}) & 2(z_{1}+z_{2}-2z_{1}z_{2}) & -i(2-z_{1}-z_{2}) \\
2+iz_{1} & 2iz_{1} & iz_{1} \\
iz_{2} & -2iz_{2} & 2+iz_{2}
\end{bmatrix}
\begin{bmatrix}
a \\
b \\
c
\end{bmatrix}=
\begin{bmatrix}
0 \\
0 \\
0
\end{bmatrix}.
\]
The characteristic polynomial of the $3 \times 3$ matrix appearing in the above equation is 
\[
4(z_{1}+z_{2}-2)^{2}-4i (z_{1}+z_{2}-2z_{1}z_{2})^{2} +16 (z_{1}z_{2}-1) +16iz_{1}z_{2}(z_{1}z_{2}-1). 
\]
When $z_{1}=z_{2}=z$, the above expression becomes 
\[
32i z(z-1)(z-i). 
\]
Since $z \in S^{1}$, the above does not vanish when $z \neq 1,i$. 
Hence $U(F_{4})$ does not have eigenvalue $i$. 
Finally, for $\psi=\,\!^{t}(a+b+c,a,b,c) \in \ecal(-1)^{\perp}$, $\eta(-1;z)\psi \in \ecal(-1)$ if and only if 
\[
\begin{bmatrix}
(1+z_{1})(1-z_{2}) & 2(1-z_{1}z_{2}) & (1+z_{1})(1-z_{2}) \\
1-i & 2(1+z_{1}) & 1+i+2z_{1} \\
1-i & 2(1+z_{2}) & 1+i+2z_{2}
\end{bmatrix}
\begin{bmatrix}
a \\
b \\
c
\end{bmatrix}=
\begin{bmatrix}
0 \\
0 \\
0
\end{bmatrix}.
\]
The characteristic polynomial of the $3 \times 3$ matrix appearing in the above equation is 
\[
-4(1-i)(z_{2}-z_{1})^{2}. 
\]
This is zero if and only if $z_{1}=z_{2}$. Therefore, $U(F_{4})$ does not have eigenvalue $-1$.  
Hence $U(F_{4})$ has no eigenvalues, and the first part in Theorem $\ref{Main3}$ in Section $\ref{INTRO}$ 
is proved. The second part of Theorem $\ref{Main3}$ follows from Corollary 1.4  in \cite{Ta1}. \hfill$\blacksquare$

\section{Quantum walks of Grover type}\label{GTYPE}

In this section, we consider the unitary time evolution $U(C)=U(S,\pcal,C)$ of the quantum walks of Grover type. 
Thus, the coin matrix $C$ is assumed to be a $D \times D$ unitary matrix satisfying $C^{2}=I$ and $C$ is 
not a scalar multiple of the identity matrix. In this case, $\sigma(C)=\{\pm 1\}$. 
Therefore, only $\pm 1$ can be eigenvalues of $U(C)$. 
Let $\ecal(\pm 1)$ be the 
eigenspace of $C$ with eigenvalue $\pm 1$. We note that since $C$ is not a 
scalar multiple of the identity, $\ecal(\pm 1) \neq \{0\}$. 
Let $\pi_{\pm}$ be the orthogonal projection onto $\ecal(\pm 1)$. 
We define a subset $E$ in the $d$-dimensional torus $T^{d}$ by 
\[
E=\bigcup_{\alpha \in S}E_{\alpha},\quad E_{\alpha}=\{z \in T^{d} \mid z^{\alpha}=-1\}. 
\]
We emphasize that we do not assume that $S$ does not contain the origin. 
However, we have $E_{0}=\emptyset$ in the definition of $E$ and hence 
$E$ is still a finite union of closed submanifolds in $T^{d}$. 
We set
\begin{equation}\label{FUTOG}
\eta(z) = \sum_{\alpha \in S}\frac{1-z^{\alpha}}{1+z^{\alpha}} P_{\alpha} \quad (z \in T^{d} \setminus E).
\end{equation}
\begin{thm}\label{Gro1}
Suppose that $U(C)$ is a quantum walk of Grover type. 
Then the following two statements are equivalent. 
\begin{enumerate}
\item $U(C)$ has eigenvalue $1$. 
\item For each $z \in T^{d} \setminus E$, there exists a non-zero vector $\psi \in \ecal(1)$ 
such that $\eta(z)\psi \in \ecal(-1)$. 
\end{enumerate}
The following two statements are also equivalent. 
\begin{enumerate}
\item[(3)] $U(C)$ has eigenvalue $-1$. 
\item[(4)] For each $z \in T^{d} \setminus E$, there exists a non-zero vector $\psi \in \ecal(-1)$ 
such that $\eta(z)\psi \in \ecal(1)$. 
\end{enumerate}
\end{thm}
Theorem $\ref{Gro1}$ is very similar to Proposition $\ref{gen1}$ at first grance. 
However, there is a difference. Indeed, in Theorem $\ref{Gro1}$, 
the equation similar to $\eqref{eigenE1}$ will be solved for $\pi_{\omega}^{\perp}$, 
and this enables us to handle a ``lazy term'' $P_{0}$ when the origin $0$ is contained in the set of steps. 
\begin{proof}
Since the proof of the equivalence of the conditions $(3)$ and $(4)$ is the same 
as that of the equivalence of $(1)$ and $(2)$, we only give the proof of the latter. 
First suppose that the condition $(1)$ holds. Then, for any $z \in T^{d}$, $\wh{C}(z)$ has eigenvalue $1$. 
We fix $z \in T^{d} \setminus E$. We take an eigenvector $\phi$ of $\wh{C}(z)$ with eigenvalue $1$. 
Since $\phi=\pi_{+}\phi + \pi_{-}\phi$, we have 
\[
\begin{split}
0  = \wh{C}(z)\phi-\phi & =\sum_{\alpha \in S} z^{\alpha} P_{\alpha}\pi_{+}\phi - \sum_{\alpha \in S} z^{\alpha} P_{\alpha} \pi_{-}\phi 
-\sum_{\alpha \in S}P_{\alpha} \pi_{+}\phi -\sum_{\alpha \in S} P_{\alpha} \pi_{-} \phi \\
& = -\sum_{\alpha \in S}
P_{\alpha}\Big[
(1-z^{\alpha}) \pi_{+}\phi +(1+z^{\alpha}) \pi_{-}\phi
\Big]. 
\end{split}
\]
Applying $P_{\alpha}$ to the above equation, we have 
\[
P_{\alpha}\pi_{-}\phi=-\frac{1-z^{\alpha}}{1+z^{\alpha}}P_{\alpha}\pi_{+}\phi. 
\]
Summing the above over all $\alpha \in S$ then gives us 
\[
\pi_{-}\phi=-\eta(z) \pi_{+}\phi. 
\]
From this equation, we see that $\pi_{+}\phi \neq 0$, $\pi_{+}\phi \in \ecal(1)$ and $\eta(z)\pi_{+}\phi \in \ecal(-1)$, showing (2). 
Conversely, suppose that the condition $(2)$ holds. We take $\psi \in \ecal(1)$ such that $\psi \neq 0$ and $\eta(z)\psi \in \ecal(-1)$. 
We set $\phi(z)=\psi -\eta(z)\psi$ for each $z \in T^{d} \setminus E$. Then, for any $z \in T^{d} \setminus E$, we see 
\[
\wh{C}(z)\phi(z)=\sum_{\alpha \in S} z^{\alpha}P_{\alpha}(\psi+\eta(z)\psi)
=\sum_{\alpha \in S} \frac{2z^{\alpha}}{1+z^{\alpha}} P_{\alpha}\psi. 
\]
We have 
\[
1-\frac{1-z^{\alpha}}{1+z^{\alpha}}=\frac{2z^{\alpha}}{1+z^{\alpha}}.  
\]
Therefore we have $\wh{C}(z)\phi(z)=\phi(z)$. Thus $\wh{C}(z)$ has eigenvalue $1$ for each $z \in T^{d} \setminus E$. 
The characteristic polynomial of $\wh{C}(z)$ is a continuous function on $T^{d}$ and $T^{d} \setminus E$ is dense in $T^{d}$. 
Hence $\wh{C}(z)$ has eigenvalue $1$ for any $z \in T^{d}$, which shows the statement $(1)$. 
\end{proof}
The necessary and sufficient condition for $U(C)$ to have an eigenvalue given in Theorem $\ref{Gro1}$ 
is strong enough. However the condition is a property of the operator $\eta(z)$ {\it for each} $z \in T^{d} \setminus E$ 
and it would not be so easy to check it. 
But it can be used to give the following much more effective sufficient condition under some more conditions. 
\begin{thm}\label{GroT}
Let $S \subset \mb{Z}^{d}$ be a finite set symmetric about the origin, 
that is, for any $\alpha$ in $S$, $-\alpha$ is also in $S$. 
Let $\{P_{\alpha}\}_{\alpha \in S}$ be a resolution of unity parametrized by $S$. 
If $\|P_{\alpha}\phi\|_{\mb{C}^{D}}=\|P_{-\alpha}\phi\|_{\mb{C}^{D}}$ holds for any $\phi \in \ecal(1)$ and $\alpha \in S$,  
then $U(C)$ has the eigenvalue $1$. 
Similarly, if $\|P_{\alpha}\phi\|_{\mb{C}^{D}}=\|P_{-\alpha}\phi\|_{\mb{C}^{D}}$ holds for any $\phi \in \ecal(-1)$ and 
$\alpha \in S$, then $U(C)$ has the eigenvalue $-1$. 
\end{thm}
\begin{proof}
Let $S_{o}$ be a subset of $S$ such that $S=S_{o} \sqcup (-S_{o}) \sqcup \{0\}$. 
Since the coefficient of $P_{0}$ in the definition $\eqref{FUTOG}$ of $\eta(z)$ is zero, we see 
\[
\eta(z)=\sum_{\alpha \in S_{o}} \frac{1-z^{\alpha}}{1+z^{\alpha}}P_{\alpha}+
\sum_{\alpha \in -S_{o}} \frac{1-z^{\alpha}}{1+z^{\alpha}}P_{\alpha}
=\sum_{\alpha \in S_{o}} \frac{1-z^{\alpha}}{1+z^{\alpha}}(P_{\alpha}-P_{-\alpha}).
\]
From this we obtain 
\begin{equation}\label{TT1}
\ispa{\eta(z)\phi,\phi}_{\mb{C}^{D}}=
\sum_{\alpha \in S_{o}}\frac{1-z^{\alpha}}{1+z^{\alpha}} 
\left(
\|P_{\alpha}\phi\|_{\mb{C}^{D}}^{2} -\|P_{-\alpha}\phi\|_{\mb{C}^{D}}^{2}
\right). 
\end{equation}
We note that the equation $\eqref{TT1}$ holds for any $\phi \in \mb{C}^{D}$ if $S$ is symmetric about the origin. 
If we have $\|P_{\alpha}\phi\|_{\mb{C}^{D}}=\|P_{-\alpha}\phi\|_{\mb{C}^{D}}$ for any $\alpha \in S$ and $\phi \in \ecal(1)$, 
we see that $\ispa{\eta(z)\phi,\phi}_{\mb{C}^{D}}=0$ for $\phi \in \ecal(1)$. By a polarization identity, we have 
$\ispa{\eta(z)\phi,\psi}_{\mb{C}^{D}}=0$ for any $\phi,\psi \in \ecal(1)$. 
This shows that $\eta(z) \ecal(1) \subset \ecal(-1)=\ecal(1)^{\perp}$. 
Therefore Theorem $\ref{Gro1}$ shows the first part of Theorem $\ref{GroT}$. 
The second part follows similarly from $\eqref{TT1}$ and Theorem $\ref{Gro1}$. 
\end{proof}

When $U(S,\pcal,C)$ is of reflection type, 
the coin matrix $C$ has the form
\begin{equation}\label{RL1}
C\phi=C_{\mu}\phi=2\ispa{\phi,\mu}_{\mb{C}^{D}}\mu-\phi \quad (\phi \in \mb{C}^{D}), 
\end{equation}
where $\mu$ is a fixed unit vector in $\mb{C}^{D}$ which spans 
the eigenspace of $C_{\mu}$ corresponding to the eigenvalue $1$. 
In this case we have a precise description of the eigenspace corresponding 
to the eigenvalue $1$ of $U(S,\pcal,C_{\mu})$. 
\begin{thm}\label{EGS1}
Let us suppose that $S$ is symmetric about the origin and that 
$\|P_{\alpha}\mu\|_{\mb{C}^{D}}=\|P_{-\alpha}\mu\|_{\mb{C}^{D}}$ holds for any $\alpha \in S$. 
Then the eigenspace of $U(S,\pcal,C_{\mu})$ corresponding to the eigenvalue $1$ is given by 
\[
\ell^{2}(\mb{Z}^{d}) \ast w=\{f \ast w \mid f \in \ell^{2}(\mb{Z}^{d})\}, 
\]
where $\ell^{2}(\mb{Z}^{d})$ denotes the $\ell^{2}$-space of $\mb{C}$-valued functions on $\mb{Z}^{d}$,  
the function $w$ is the inverse Fourier transform of $w_{o} \in L^{2}(T^{d},\mb{C}^{D})$ given by 
\[
w_{o}(z)=\frac{\mu -\eta(z)\mu}{\|\mu-\eta(z)\mu\|_{\mb{C}^{D}}}
=\left(
\sum_{\alpha \in S} \frac{\|P_{\alpha}\mu\|_{\mb{C}^{D}}^{2}}{|1+z^{\alpha}|^{2}}
\right)^{-1/2}
\sum_{\alpha \in S}\frac{z^{\alpha}}{1+z^{\alpha}}P_{\alpha}\mu \quad (z \in T^{d} \setminus E), 
\]
and the convolution $f \ast w$ of $f \in \ell^{2}(\mb{Z}^{d})$ and $w \in \ell^{2}(\mb{Z}^{d},\mb{C}^{D})$ is 
given by 
\[
(f \ast w)(x)=\sum_{y \in \mb{Z}^{d}}f(x-y)w(y). 
\]
\end{thm}
\begin{proof}
Let $\ucal=\ucal(S,\pcal,C_{\mu})$ be a unitary operator given in $\eqref{FOP1}$, $\eqref{FOP2}$. 
Under the assumption in the statement, $1$ is an eigenvalue of $\ucal$. 
Let $k \in L^{2}(T^{d},\mb{C}^{D})$ be an eigenfunction of $\ucal$ with eigenvalue $1$. 
Then $\wh{C_{\mu}}(z)k(z)=k(z)$ for all $z \in T^{d}$. 
We express $k(z)$ as 
\[
k(z)=g(z)\mu +\pi_{-}k(z),\quad g(z)=\ispa{k(z),\mu}_{\mb{C}^{D}}.  
\]
Both of $g(z)$ and $\pi_{-}k(z)$ are $L^{2}$-functions. 
As in the proof of Theorem $\ref{Gro1}$, we see $\pi_{-}k(z)=-g(z)\eta(z)\mu$, and hence 
\[
k(z)=g(z)(\mu -\eta(z)\mu)=g(z)\|\mu-\eta(z)\mu\|_{\mb{C}^{D}} w_{o}(z). 
\]
Since $\|k(z)\|_{\mb{C}^{D}}^{2}=|g(z)|^{2}\|\mu-\eta(z)\mu\|_{\mb{C}^{D}}^{2}$, 
the function $f(z)=g(z)\|\mu-\eta(z)\mu\|_{\mb{C}^{D}}$ is an $L^{2}$-function on $T^{d}$. 
Thus, $k \in L^{2}(T^{d}) w_{o}$. Conversely, since $\|w_{o}(z)\|_{\mb{C}^{D}}=1$ for any $z \in T^{d}$, 
we have $fw_{o} \in L^{2}(T^{d},\mb{C}^{D})$ for 
any $f \in L^{2}(T^{d})$, and it is straightforward to see that $fw_{o}$ is 
an eigenfunction of $\ucal$ with eigenvalue $1$. Hence $L^{2}(T^{d})w_{o}$ 
is the eigenspace of $\ucal$ corresponding to the eigenvalue $1$. 
It is easy to show that $\ell^{2}(\mb{Z}^{d}) \ast w \subset \ell^{2}(\mb{Z}^{d},\mb{C}^{D})$ and 
\[
\fcal (\ell^{2}(\mb{Z}^{d}) \ast w)=L^{2}(T^{d}) w_{o}. 
\] 
This shows that $\ell^{2}(\mb{Z}^{d}) \ast w$ is the eigenspace of $U(S,\pcal,C)$ 
with the eigenvalue $1$. 
\end{proof}

We have the following criterion for quantum walks of reflection type not to have eigenvalue $-1$. 
\begin{thm}\label{EVN1}
Let $S \subset \mb{Z}^{d}$ be a set of steps containing the origin, 
and let $\pcal=\{P_{\alpha}\}_{\alpha \in S}$ be a resolution of unity on $\mb{C}^{D}$ such that ${\rm rank}\, P_{0}=1$. 
Let $\mu \in \mb{C}^{D}$ be a unit vector and let $U(C_{\mu})=U(S,\pcal,C_{\mu})$ be the quantum walk of reflection type 
with coin matrix $C_{\mu}$ given by $\eqref{RL1}$. 
If $P_{0}\mu \neq 0$, then $U(C_{\mu})$ does not have eigenvalue $-1$. 
\end{thm}
\begin{proof}
Let us suppose that $P_{0} \mu \neq 0$ and $U(C_{\mu})$ has the eigenvalue $-1$. We take and fix $z \in T^{d}$ 
such that $z^{\alpha} \neq \pm 1$ for any $\alpha \in S$, $\alpha \neq 0$. 
Then, by Theorem $\ref{Gro1}$, there exists a non-zero vector $\psi \in \mb{C}^{D}$ 
and a constant $c(z) \in \mb{C}$ such that $\ispa{\psi,\mu}_{\mb{C}^{D}}=0$ and $\eta(z)\psi=c(z)\mu$. 
Since the coefficient of $P_{0}$ in $\eta(z)$ is zero, we have $c(z)P_{0}\mu=0$ and hence $c(z)=0$ because $P_{0}\mu \neq 0$. 
This shows that $\eta(z)\psi=0$. Since $z^{\alpha} \neq \pm 1$ for any $0 \neq \alpha \in S$, 
we conclude $P_{\alpha}\psi=0$ for any $\alpha \neq 0$. 
Thus, $\psi=P_{0}\psi$. By the assumption that ${\rm rank}\, P_{0}=1$, there exists a constant $c$ such 
that $\psi=cP_{0}\mu$. Then $0=\ispa{\psi,\mu}_{\mb{C}^{D}}=c\|P_{0}\mu\|^{2}_{\mb{C}^{D}}$, which shows that $c=0$ and 
hence $\psi=0$, a contradiction. 
\end{proof}

When $P_{0}\mu=0$, $U(C_{\mu})=U(S,\pcal,C_{\mu})$ could have eigenvalue $-1$ as shown in the following theorem. 
\begin{thm}\label{EVNN}
We consider a lazy quantum walk $U(C_{\mu})=U(S_{{\rm lazy}}, \pcal_{{\rm lazy}}, C_{\mu})$ where 
the unit vector $\mu=\,\!^{t}(a_{1},\ldots,a_{2d+1}) \in \mb{C}^{2d+1}$ satisfies $|a_{j}|=|a_{d+1+j}|$ 
for every $j=1,\ldots,d$. 
We suppose that $a_{d+1}=0$. Then, the unitary operator $U(C_{\mu})$ has both of the eigenvalue $\pm 1$. 
\end{thm}
\begin{proof}
The assumption on $\mu$ ensures that $U(C_{\mu})$ has eigenvalue $1$ by Theorem $\ref{GroT}$. 
We use Theorem $\ref{Gro1}$ to prove the assertion. 
In the setting described in the statement, the subset $E$ of $T^{d}$ is given by 
\begin{equation}\label{exD}
E=\{z=(z_{1},\ldots,z_{d}) \in T^{d} \mid z_{j}=-1 \ \mbox{for some $j=1,\ldots,d$}\}. 
\end{equation}
We take $z \in T^{d} \setminus E$. We set $J(z)=\{j \mid j=1,\ldots,d,\,z_{j}=1\}$. 
When $J(z)=\emptyset$, we define $\psi=\,\!^{t}(\psi_{1},\ldots,\psi_{2d+1}) \in \mb{C}^{2d+1}$ by 
$\psi_{d+1}=0$ and 
\begin{equation}\label{psi1}
\psi_{j}=\frac{1+z_{j}}{1-z_{j}} a_{j},\quad 
\psi_{d+1+j}=-\frac{1+z_{j}}{1-z_{j}} a_{d+1+j} \quad (j=1,\ldots,d). 
\end{equation}
Since $\mu$ is a unit vector, some of $a_{j}$ ($j=1,\ldots,d$) do not vanish. Hence $\psi \neq 0$. 
From $\eqref{psi1}$, we have 
\[
\ispa{\psi,\mu}_{\mb{C}^{2d+1}}=\sum_{j=1}^{d} (\psi_{j} \ol{a_{j}} + \psi_{d+1+j} \ol{a_{d+1+j}})
=\sum_{j=1}^{d}
\frac{1+z_{j}}{1-z_{j}}
\left(
|a_{j}|^{2}-|a_{d+1+j}|^{2}
\right)=0, 
\]
which shows that $0 \neq \psi \in \ecal(-1)$. Since $a_{d+1}=0$, we see 
\[
\eta(z)\psi=\sum_{j=1}^{d}
\frac{1-z_{j}}{1+z_{j}}
\left(
\psi_{j} \pmb{e}_{j} -\psi_{d+1+j}\pmb{e}_{d+1+j}
\right)
=\sum_{j=1}^{d}(a_{j}\pmb{e}_{j}+a_{d+1+j}\pmb{e}_{d+1+j})=\mu, 
\]
and hence $\eta(z)\psi \in \ecal(1)$. 

When $J(z) \neq \emptyset$, we set $\psi_{j}=\psi_{d+1+j}=0$ for $j \not\in J(z)$. 
We further define 
\[
\psi_{j}=a_{j},\quad \psi_{d+1+j}=-a_{d+1+j} \quad (j \in J(z)). 
\]
Finally, we set $\psi_{d+1}=1$. Then, $\psi \neq 0$ and 
we have 
\[
\ispa{\psi,\mu}_{\mb{C}^{2d+1}}=\sum_{j \in J(z)} \left(
\psi_{j} \ol{a_{j}}+\psi_{d+1+j} \ol{a_{d+1+j}}
\right)=
\sum_{j \in J(z)} \left(
|a_{j}|^{2}-|a_{d+1+j}|^{2}
\right)=0, 
\]
and hence $\psi \in \ecal(-1)$. We also have 
\[
\eta(z)\psi=\sum_{j \not\in J(z)}\frac{1-z_{j}}{1+z_{j}}(\psi_{j}\pmb{e}_{j}+\psi_{d+1+j}\pmb{e}_{d+1+j})=0
\]
by the definition of $\psi$. This shows that $0=\eta(z)\psi \in \ecal(1)$. Hence in this case 
the operator $U(C_{\mu})$ has eigenvalue $-1$ by Thoerem $\ref{Gro1}$. 
\end{proof}

\section{Examples}\label{EXA}

In this section, we discuss various examples where we can apply results in the previous sections. 
Theorems $\ref{Main1}$, $\ref{Main2}$ in Section $\ref{INTRO}$ will be proved in this section. 

\subsection{Grover walk on $\pmb{\ell^{2}(\mb{Z}^{d},\mb{C}^{2d})}$}\label{GroverZd}

This subsection is devoted to prove Theorem $\ref{Main1}$ in Section $\ref{INTRO}$. 
Let us consider the Grover walk $U(G_{2d})=U(S_{{\rm std}}, \pcal_{{\rm std}}, G_{2d})$ 
on $\ell^{2}(\mb{Z}^{d},\mb{C}^{2d})$, whose concrete form is given in $\eqref{UQW}$, 
where $S_{{\rm std}}$, $\pcal_{{\rm std}}$ is given in the paragraph containing $\eqref{std1}$, 
and the coin matrix $G_{2d}$ is given in $\eqref{GroC}$. 
For any positive integer $D$, we set 
\begin{equation}\label{GroV}
\mu_{D}=\frac{1}{\sqrt{D}}\,\!^{t}(1,\ldots,1) \in \mb{C}^{D}.  
\end{equation}
Then $G_{2d}=C_{\mu_{2d}}$ where $C_{\mu_{2d}}$ is given in $\eqref{RL1}$ with $\mu=\mu_{2d}$. 
Thus we have $\ecal(1)=\mb{C}\mu_{2d}$ and $\ecal(-1)=(\mb{C}\mu_{2d})^{\perp}=\mu_{2d}^{\perp}$. 
Since $\dsp \|P_{\pm u_{j}}\mu_{2d}\|_{\mb{C}^{2d}}=\frac{1}{\sqrt{2d}}$, 
we have $\|P_{\alpha}\mu_{2d}\|_{\mb{C}^{2d}}=\|P_{-\alpha}\mu_{2d}\|_{\mb{C}^{2d}}$ 
for any $\alpha \in S_{{\rm std}}$ and 
hence $U(G_{2d})$ has eigenvalue $1$ by Theorem $\ref{GroT}$. 

Next, let us consider whether $U(G_{2d})$ has eigenvalue $-1$ or not. 
Since Theorem $\ref{GroT}$ gives only a sufficient condition for $U(G_{2d})$ to have eigenvalue $-1$, 
let us use Theorem $\ref{Gro1}$. In this case, we have 
\[
\eta(z)=\sum_{j=1}^{d}\frac{1-z_{j}}{1+z_{j}}(P_{2j-1}-P_{2j}).
\]
Let us take $\psi \in \mb{C}^{2d}$ and write $\psi=\,\!^{t}(\psi_{1},\ldots,\psi_{2d})$. 
Then $\psi \in \ecal(-1)=\mu_{2d}^{\perp}$ if and only if $\psi_{1}+\cdots +\psi_{2d}=0$ 
by the definition of the vector $\mu_{2d}$. The subset $E$ in $T^{d}$ is given in $\eqref{exD}$. 
Now, let us take $z=(z_{1},\ldots,z_{d}) \in T^{d} \setminus E$. 
We set $J(z)=\{j \mid 1 \leq j \leq d,\, z_{j}=1\}$. 
When $J(z) \neq \emptyset$, we define $\psi=(\psi_{1},\ldots,\psi_{2d}) \in \mb{C}^{2d}$ by 
\[
\psi_{2j-1}=-\psi_{2j}=1 \ (j \in J(z)),\quad \psi_{2j-1}=\psi_{2j}=0 \ (j \not\in J(z)). 
\]
Then $\psi \neq 0$, $\psi \in \ecal(-1)$ and $\eta(z)\psi=0$.
When $J(z)=\emptyset$, we set 
\[
\psi_{2j-1}=-\psi_{2j}=\frac{1+z_{j}}{1-z_{j}} \ \ (j=1,\ldots,d). 
\]
Then, it is clear that $\psi \neq 0$, $\psi \in \ecal(-1)$ and $\eta(z)\psi=\sqrt{2d}\mu_{2d} \in \ecal(1)$. 
Therefore, by Theorem $\ref{Gro1}$, we conclude that $U(G_{2d})$ has eigenvalue $\pm 1$. 
Now, $\eqref{LTA1}$ of Theorem $\ref{Main1}$ in Section $\ref{INTRO}$ follows from the Wiener formula $\eqref{wie}$. 
\hfill$\blacksquare$

\subsection{Lazy Grover walk on $\pmb{\ell^{2}(\mb{Z}^{d},\mb{C}^{2d+1})}$}\label{GroverL}

This subsection is devoted to prove Theorem $\ref{Main2}$. 
We consider lazy Grover walk 
$U(G_{2d+1})=U(S_{{\rm lazy}}, \pcal_{{\rm lazy}}, G_{2d+1})$, whose concrete form is given in $\eqref{lazyD}$, 
where $\pcal_{{\rm lazy}}$ is given in the paragraph containing $\eqref{lazy1}$. 
We note that the Grover coin matrix $G_{2d+1}$ is still equal to $C_{\mu_{2d+1}}$ where $\mu_{2d+1}$ is 
given in $\eqref{GroV}$. By the same argument as in Subsection $\ref{GroverZd}$, 
it is shown that $U(G_{2d+1})$ has eigenvalue $1$. 
However, in this case, $U(G_{2d+1})$ do not have eigenvalue $-1$. 
Indeed, in this case we have ${\rm rank}\,P_{0}=1$ and $\dsp P_{0}\mu=\frac{1}{\sqrt{2d+1}}\pmb{e}_{d+1} \neq 0$. 
Hence Theorem $\ref{EVN1}$ shows that $U(G_{2d+1})$ does not have eigenvalue $-1$. 

Since $\sigma(G_{2d+1})=\{\pm 1\}$, the lazy Grover walk $U(G_{2d+1})$ has only eigenvalue $1$. 
The second part of Theorem $\ref{Main2}$ follows from Corollary 1.5 in \cite{Ta1}. 
\hfill$\blacksquare$

\subsection{Deformation of Grover walks}\label{Defo}

In \cite{WKKK}, a homogeneous quantum walks with coin matrices given by 
\begin{equation}\label{WPF1}
\begin{pmatrix}
-p & q & \sqrt{pq} & \sqrt{pq} \\
q & -p & \sqrt{pq} & \sqrt{pq} \\
\sqrt{pq} & \sqrt{pq} & -q & p \\
\sqrt{pq} & \sqrt{pq} & p & -q
\end{pmatrix}
\end{equation}
with $0<p=1-q<1$ was investigated. In particular the limit distributions of the quantum walks defined by the 
coin matrix $\eqref{WPF1}$ are computed in \cite{WKKK}. 
The coin matrix $\eqref{WPF1}$ is written as $C_{\mu(p)}$ where $\mu(p)$ is given by 
\[
\mu(p)=\frac{1}{\sqrt{2}}\,\!^{t}(\sqrt{q},\sqrt{q},\sqrt{p},\sqrt{p}). 
\]
Then by Theorem $\ref{GroT}$ it is proved in a similar way as in Subsection $\ref{GroverZd}$ 
that $U(S_{{\rm std}}, \pcal_{{\rm std}}C_{\mu(p)})$ (on $\ell^{2}(\mb{Z}^{2},\mb{C}^{4})$) has both eigenvalue $\pm 1$. 
We remark that the coin matrix $\eqref{WPF1}$ is a one-parameter deformation of the Grover coin matrix $G_{4}$,  
because we have $G_{4}=C_{\mu(1/2)}$. Such a deformation was further considered in \cite{SBJ}. Indeed in \cite{SBJ} 
a dynamical behavior, such as peak velocities, of the quantum walk with the coin matrix
\[
C_{2}(\rho)=
\begin{pmatrix}
-\rho^{2} & \rho\sqrt{2(1-\rho^{2})} & 1-\rho^{2} \\
\rho \sqrt{2(1-\rho^{2})} & 2\rho^{2}-1 & \rho \sqrt{2(1-\rho^{2})} \\
1-\rho^{2} & \rho \sqrt{2(1-\rho^{2})} & -\rho^{2}
\end{pmatrix}
\]
is investigated. The coin matrix $C_{2}(\rho)$ can by written as 
\[
C_{2}(\rho)=C_{\mu(\rho)},\quad \mu(\rho)=v_{3}(\rho)=\,\!^{t}\!\left(
\sqrt{\frac{1-\rho^{2}}{2}}, \, \rho, \, \sqrt{\frac{1-\rho^{2}}{2}}
\right), 
\]
and it satisfies the condition in Theorem $\ref{GroT}$ for eigenvalue $1$. 
Therefore, the unitary operator $U(C_{2}(\rho))=U(S_{{\rm lazy}}, \pcal_{{\rm lazy}}, C_{2}(\rho))$ in 
one dimension has eigenvalue $1$. 
It is mentioned in \cite{SBJ} that $U(C_{2}(\rho))$ do not have eigenvalue $-1$. 
Indeed, when $\rho \neq 0$, this follows from Theorem $\ref{EVN1}$. 
But for $\rho=0$, straightforward computation shows that the unitary operator $U(C_{2}(0))$ has the eigenvalues $\pm 1$.
 
These deformations of Grover walks can be generalized. 
Let us explain it for lazy quantum walks $U(C_{\mu})=U(S_{{\rm lazy}}, \pcal_{{\rm lazy}}, C_{\mu})$ of reflection type 
with a unit vector $\mu \in \mb{C}^{2d+1}$. 
Let us suppose that the unit vector $\mu=\,\!^{t}(a_{1},\ldots,a_{2d+1})$ in $\mb{C}^{2d+1}$ satisfies 
$|a_{j}|=|a_{d+1+j}|$ for $j=1,\ldots,d$. 
This assumption is nothing but the condition 
for eigenvalue $1$ in Theorem $\ref{GroT}$. 
Then, the vectors 
\[
\nu=\,\!^{t}(a_{1},a_{2},\ldots,a_{d}),\quad 
\tilde{\nu}=\,\!^{t}(a_{d+2},a_{d+3},\ldots,a_{2d+1}),
\]
have the same norm, and we can set $a_{d+1+j}=e^{\sqrt{-1}\theta_{j}}a_{j}$ for some $\theta_{j} \in \mb{R}$. 
For simplicity, we write 
\[
p=\|\nu\|,\quad \rho=|a_{d+1}|=\sqrt{1-2p^{2}},\quad 
a_{d+1}=\rho e^{\sqrt{-1}\varphi}. 
\]
Denoting the $d \times d$ diagonal unitary matrix ${\rm diag}(e^{\sqrt{-1}\theta_{1}},\ldots, e^{\sqrt{-1}\theta_{d}})$ by $D$, 
we have $\tilde{\nu}=D\nu$ and $\mu=\,\!^{t}[\nu, \rho e^{\sqrt{-1} \varphi}, D\nu]$. 
If $\nu=0$, then $\mu=\rho e^{\sqrt{-1}\varphi}\pmb{e}_{d+1}$. In this case, $U(C_{\mu})$ is rather easy to handle 
because $\wh{C_{\mu}}(z)$ is a diagonal matrix and the corresponding unitary operator $U(C_{\mu})$ has only 
eigenvalue $1$. Thus, we may suppose that $\nu \neq 0$. 
Since $\nu$ and $\dsp \nu_{0}=p\mu_{d}=\frac{p}{\sqrt{d}}(1,\ldots,1) \in \mb{C}^{d}$ lie on the $2d-1$-dimensional 
sphere of radius $p$, there exists a smooth curve $\nu(t)$ ($t \in [0,1]$) in $\mb{C}^{d}$ such that 
\begin{equation}\label{spath1}
\|\nu(t)\|=p \ \ (t \in [0,1]),\quad \nu(0)=\nu_{0},\quad \nu(1)=\nu. 
\end{equation}
For example, we can take, as a curve $\nu(t)$, a geodesic (a part of a great circle) on the $2d-1$-dimensional 
sphere joining $\nu_{0}$ and $\nu$. 
For $t \in [0,1]$, we define the $d \times d$ unitary diagonal matrix $D(t)$ by 
\begin{equation}\label{spath2}
D(t)={\rm diag}(e^{\sqrt{-1}t \theta_{1}},\ldots,e^{\sqrt{-1}t \theta_{d}}) \quad 
\mbox{so that} \quad D(0)=I_{d} \quad D(1)=D.  
\end{equation}
Let $f(t)$ be an arbitrary (smooth) real-valued function in $t \in [0,1]$ such that 
\begin{equation}\label{spath3}
0 \leq f(t) \leq \frac{1}{p\sqrt{2}} \ (t \in [0,1]) \quad f(0)=\frac{1}{p}\sqrt{\frac{d}{2d+1}},\quad f(1)=1. 
\end{equation}
Finally, we define a function $\rho(t)$ in $t \in [0,1]$ by 
\begin{equation}\label{spath4}
\rho(t)=\sqrt{1-2p^{2}f(t)^{2}} \quad \mbox{so that} \quad 
\rho(0)=\frac{1}{\sqrt{2d+1}},\quad \rho(1)=\sqrt{1-2p^{2}}=\rho. 
\end{equation}
Then, the vector $\mu(t)$ ($t \in [0,1]$) defined by 
\begin{equation}\label{spath5}
\mu(t)=
\begin{bmatrix}
f(t)\nu(t) \\
\rho(t)e^{\sqrt{-1}t\varphi} \\
f(t) D(t) \nu(t)
\end{bmatrix}
\in \mb{C}^{2d+1}
\end{equation}
satisfies 
\[
\mu(0)=\mu_{2d+1}=\frac{1}{\sqrt{2d+1}}(1,\ldots,1),\quad 
\mu(1)=\mu,\quad \|\mu(t)\|_{\mb{C}^{2d+1}}=1 \ (t \in [0,1]). 
\]
Therefore, $C_{\mu(t)}$ intertwines $C_{\mu}$ with $G_{2d+1}$. 
By the construction, it is clear from Theorem $\ref{GroT}$ that 
the corresponding lazy quantum walk $U(C_{\mu(t)})=U(S_{{\rm lazy}}, \pcal_{{\rm lazy}}, C_{\mu(t)})$ has 
the eigenvalue $1$ for any $t \in [0,1]$. 
If $(d+1)$-component of $\mu(t)$ does not vanish, then $U(C_{\mu(t)})$ does not have $-1$ as an eigenvalue. 
However, if $(d+1)$-component of $\mu(t)$ vanishes, then $U(C_{\mu(t)})$ has eigenvalue $-1$ by Theorem $\ref{EVNN}$. 
In particular, when $\rho=0$, which means $a_{d+1}=0$, $U(C_{\mu(1)})=U(C_{\mu})$ has both of 
eigenvalue $\pm 1$, but, when $f(t)<1$ for $t<1$, $U(C_{\mu(t)})$ does not have eigenvalue $-1$ 
because $(d+1)$-component of $\mu(t)$ ($t<1$) does not vanish. 
Summarizing the above argument, we have obtained the following proposition. 
\begin{prop}\label{deform1}
For any unit vector $\mu=\,\!^{t}(a_{1},\ldots,a_{2d+1}) \in \mb{C}^{2d+1}$ satisfying 
$|a_{j}|=|a_{d+1+j}|$ for $j=1,\ldots,d$. 
Let us assume also that $a_{d+1}=0$. 
Then there is a one-parameter familiy of 
unit vectors $\mu(t)$ $(t \in [0,1])$ satisfying the following properties. 
\begin{enumerate}
\item $\mu(t)$ is smooth in $t \in [0,1)$. 
\item $\mu(0)=\mu_{2d+1}$ and $\mu(1)=\mu$. 
\item The unitary operator $U(C_{\mu(t)})=U(S_{{\rm lazy}}, \pcal_{{\rm lazy}}, C_{\mu(t)})$ has 
eigenvalue $1$ for each $t \in [0,1]$. 
\item For each $t \in [0,1)$, $U(C_{\mu(t)})$ does not have $-1$ as an eigenvalue. 
\item $U(C_{\mu(1)})=U(S_{{\rm lazy}}, \pcal_{{\rm lazy}}, C_{\mu})$ has the eigenvalue $-1$. 
\end{enumerate}
Such a one-parameter deformation $\mu(t)$ is concretely constructed 
by the procedure $\eqref{spath1}$, $\eqref{spath2}$, $\eqref{spath3}$, $\eqref{spath4}$ and $\eqref{spath5}$ 
with a choice of a smooth function $f(t)$ such that $f(t)<1$ for $t \in [0,1)$. 
\end{prop}

\subsection{The Grover walk on the triangular lattice}\label{TRI}

Our setting-up also works well for Grover walks on certain crystal lattices, 
since we can choose the set $S$ of steps and the resolution of unity $\pcal$ rather arbitrarily.  
For example, let $u_{1},u_{2}$ be the standard basis of $\mb{Z}^{2}$ and $u_{3}=(1,1)$. 
We set $S=\{\pm u_{i} \mid i=1,2,3\}$. We define $P_{u_{i}}=P_{2i-1}$ and $P_{-u_{i}}=P_{2i}$. 
Then we can consider the Grover walk $U(G_{6})=U(S,\pcal,G_{6})$ on $\ell^{2}(\mb{Z}^{2},\mb{C}^{6})$ where
$\pcal=\{P_{\alpha}\}_{\alpha \in S}$. 
The following corollary is a direct consequence of Theorem $\ref{GroT}$. 
\begin{cor}\label{TriT}
The unitary operator $U(S,\pcal,G_{6})$ defined above has both of eigenvalue $\pm 1$. 
\end{cor}

\section{Product of two quantum walks of Grover type}\label{PRO}

In this section, we discuss the eigenvalue of a product of two quantum walks of Grover type, 
namely we consider the following unitary operator
\[
U_{C}:=\scal^{*} C\scal C=\left(
\sum_{\alpha \in S}\tau^{-\alpha}P_{\alpha}C
\right)
\left(
\sum_{\alpha \in S} \tau^{\alpha} P_{\alpha}C
\right),
\]
where $C$ is a coin matrix of Grover type. 
We note that, in this section, {\it we do not assume that the set $S$ of steps is symmetric about the origin}. 
This kind of operator naturally arise when quantum walks on the triangular lattice is considered.
See Corollary $\ref{PrG}$ below. Note that the unitary operator $U_{C}$ is a PUTO on $\ell^{2}(\mb{Z}^{d},\mb{C}^{D})$, 
and the corresponding unitary-matrix valued function on the torus $T^{d}$ is given by 
\begin{equation}\label{PUTOP}
\wh{U}_{C}(z)=V(z)^{*}CV(z)C=\left(
\sum_{\alpha \in S}z^{-\alpha}P_{\alpha}C
\right)
\left(
\sum_{\alpha \in S} z^{\alpha} P_{\alpha}C
\right),  
\end{equation}
where $V(z)$ is defined in $\eqref{MV1}$. 
As before, let $\ecal(\pm 1)$ denote the eigenspace of $C$ with eigenvalue $\pm 1$, respectively, 
and let $\pi_{\pm}$ denote the orthogonal projection onto $\ecal(\pm 1)$, respectively. 
It is straightforward, using the equation $C=\pi_{+}-\pi_{-}$, to see the following. 
\begin{equation}\label{PM1}
\wh{U}_{C}(z) \pm I=\pm 2 V(z)(\pi_{\pm} V(z)^{*} \pi_{+} +\pi_{\mp}V(z)^{*} \pi_{-}). 
\end{equation}
\begin{thm}\label{PUTOPT}
Let $U_{C}$ be as above. If $\dim \ecal(1) < \dim \ecal(-1)$, then $U_{C}$ has an eigenvalue $1$. 
Hence the quantum walks defined by $U_{C}$ starting at the origin with an initial state is localized at some points. 
\end{thm}
\begin{proof}
For each fixed $z \in T^{d}$, define the subspace $\ve(z)$ in $\mb{C}^{D}$ by
\[
\ve(z)=\pi_{-}V(z)\ecal(1). 
\]
By definition and the assumption, $\ve(z)$ is a proper subspace of $\ecal(-1)$. 
For fixed $z \in T^{d}$, we choose a non-zero vector $\psi(z)\in \ecal(-1)$ such that $\psi(z) \perp \ve(z)$. 
For any $\phi \in \ecal(1)$, we see 
\[
\ispa{\phi,V(z)^{*}\psi(z)}_{\mb{C}^{D}}=\ispa{V(z)\phi,\psi(z)}_{\mb{C}^{D}}=\ispa{\pi_{-}V(z)\phi,\psi(z)}_{\mb{C}^{D}}=0, 
\]
which shows that $V(z)^{*}\psi(z) \in \ecal(1)^{\perp}=\ecal(-1)$, and hence $\pi_{+}V(z)^{*}\psi(z)=0$. 
Since $\pi_{+}\psi(z)=0$, we see $\wh{U}_{C}(z)\psi(z)-\psi(z)=0$ by $\eqref{PM1}$. Thus, $\wh{U}_{C}(z)$ has 
an eigenvalue $1$ for all $z \in T^{d}$, and hence $U_{C}$ has an eigenvalue $1$. 
\end{proof}

As an example let $S=\{u_{1},u_{2},u_{3}\}$ where $\{u_{1},u_{2}\}$ 
is the standard basis of $\mb{Z}^{2}$ and $u_{3}=(-1,-1)$. 
We define a resolution of unity $\pcal=\{P_{\alpha}\}_{\alpha \in S}$ on $\mb{C}^{3}$ 
by setting $P_{u_{i}}=P_{i}$ for $i=1,2,3$. 
Then we can consider the PUTO $U(S,\pcal,G_{3})$ defined in $\eqref{PUTOP}$ with $C=G_{3}$. 
By Theorem $\ref{PUTOPT}$, we have the following. 
\begin{cor}\label{PrG}
The unitary operator $U(S,\pcal,G_{3})$ has the eigenvalue $1$. 
\end{cor}

\end{document}